\documentclass[a4paper]{amsart}
\usepackage{amsfonts,amsmath,amssymb,amscd,amstext,amsbsy,latexsym}
\newtheorem{thm}{Theorem}[section]
\newtheorem{cor}[thm]{Corollary}

\newtheorem{prop}[thm]{Proposition}
\theoremstyle{definition}

\theoremstyle{remark}
\newtheorem{rem}[thm]{Remark}
\newtheorem{exm}[thm]{Example}
\numberwithin{equation}{section}

\begin{document}

\title{Prime Avoidance Property}%
\author{Alborz Azarang}%

\maketitle

\centerline{Department of Mathematics, Faculty of Mathematical Sciences and Computer,}
\centerline{ Shahid Chamran University
of Ahvaz, Ahvaz-Iran} \centerline{a${}_{-}$azarang@scu.ac.ir}
\centerline{ORCID ID: orcid.org/0000-0001-9598-2411}

\begin{abstract}
Let $R$ be a commutative ring, we say that $\mathcal{A}\subseteq Spec(R)$ has prime avoidance property, if $I\subseteq \bigcup_{P\in\mathcal{A}}P$ for an ideal $I$ of $R$, then there exists $P\in\mathcal{A}$ such that $I\subseteq P$. We exactly determine when $\mathcal{A}\subseteq Spec(R)$ has prime avoidance property. In particular, if $\mathcal{A}$ has prime avoidance property, then $\mathcal{A}$ is compact. For certain classical rings we show the converse holds (such as Bezout rings, $QR$-domains, zero-dimensional rings and $C(X)$). We give an example of a compact $\mathcal{A}\subseteq Spec(R)$ of a Prufer domain $R$ which has not $P.A$-property. Finally, we show that if $V,V_1,\ldots, V_n$ are valuations for a field $K$ and $V[x]\nsubseteq \bigcup_{i=1}^n V_i$ for some $x\in K$, then there exists $v\in V$ such that $v+x\notin \bigcup_{i=1}^n V_i$.

\end{abstract}

\vspace{0.5cm}
\section{Introduction}
Quentel in \cite[Proposition 9]{quen} proved that the following are equivalent for a commutative reduced ring $R$ (see also \cite[Propositions 1.4 and 1.15]{mat}):

\begin{enumerate}
\item $Q(R)$, the classical ring of quotient of $R$, is a $VNR$.
\item If $I$ is an ideal of $R$ contained in the union of the minimal prime ideals of $R$, then $I$ is contained in one of them.
\item $Min(R)$ is compact; and if a finitely generated ideal is contained in the union of the minimal prime ideals of $R$, then it is contained in one of them.
\end{enumerate}

In \cite{quen}, Quentel has produced an example of a reduced ring $R$ where $Min(R)$ is compact, but $Q(R)$ is not a $VNR$, which shows the second part of $(3)$ is necessary.\\

The main aim of this paper is to show that the conditions $(2)$ and $(3)$ of the above are equivalent for each subset $\mathcal{A}$ of $Spec(R)$, for arbitrary commutative ring $R$.\\

In commutative ideal theory one of the most important covering results is the Prime Avoidance Lemma which state that if an ideal $I$ of a commutative ring $R$ is covered by a
finite union of prime ideals $P_1,\cdots,P_n$, then there exists $i$ such that $I\subseteq P_i$, see \cite[Theorem 81]{kap}, \cite{mcadm}, \cite{mac}, \cite{shvm} and \cite{krm}. Let us denote by $V_{\mathcal{A}}(I)$ the set of all prime ideals in $\mathcal{A}$ which contains $I$, where $R$ is a commutative ring, $\mathcal{A}\subseteq Spec(R)$ and $I$ is an ideal of $R$. It is well known that the set of all $V_{\mathcal{A}}(I)$, where $I$ ranges over all ideals of $R$, is a topology for closed sets on $\mathcal{A}$, which is called hull-kernel or Zariski topology on $\mathcal{A}$ (when $\mathcal{A}=Spec(R)$, $\mathcal{A}=Max(R)$ and $\mathcal{A}=Min(R)$ we apply $V(I)$, $V_M(I)$ and $V_m(I)$, respectively). It is well known that $Spec(R)$ and $Max(R)$ are compact spaces. In fact one easily find that the fact which implies these two spaces are compact is  "each proper ideal of $R$ can be embedded in a maximal ideal" which is well known as Krull Maximal Ideal Theorem; In other words,  if $I$ is an ideal of $R$ such that $I\subseteq \bigcup_{M\in Max(R)} M$, then there exists $M\in Max(R)$ such that $I\subseteq M$. It seems that there exists a relation between the compactness of a subset $\mathcal{A}$ of $Spec(R)$ and the fact that $\mathcal{A}$ has prime avoidance property. Note that by Krull Maximal Ideal Theorem one can easily see that each subset $X$ of $Spec(R)$ which contains $Max(R)$ is compact. Moreover, if $R$ is a noetherian ring then each subset of $Spec(R)$ is compact (i.e., $Spec(R)$ is noetherian), since each sum of a family of ideals in $R$ reduced to a finite sum of the family. Note that there exist non-noetherian rings for which $Spec(R)$ is noetherian, in fact $Spec(R)$ is noetherian if and only if $ACC$ holds on radical ideals of $R$, i.e., for each ideal $I$ of $R$ there exists $a_1,\ldots, a_n$ such that $\sqrt{I}=\sqrt{(a_1,\ldots,a_n)}$; which also implies that $Min(I)$ is finite for each ideal $I$ of $R$. Note that by
\cite[Proposition 3.8]{mat}, for a ring $R$, $Min(R)=\{P_{\alpha}\}$ is finite if and only if for each $\beta$, $P_{\beta}\nsubseteq \bigcup_{\alpha\neq\beta} P_{\alpha}$, in other words for each $P\in Min(R)$, the set $Min(R)\setminus\{P\}$ has prime avoidance property (for another finiteness result for $Min(R)$ see \cite{and}). Now the following is in order.

\begin{rem}
Let $R$ be a ring. Then $Min(R)=\{P_{\alpha}\}$ is finite if and only if $Min(R)$ is compact and for each $\beta$, $Q_\beta:=\bigcap_{\alpha\neq \beta} P_{\alpha}\nsubseteq P_\beta$. In particular, if in addition $R$ is a reduced ring, then each $P_\beta$ is  not essential and $P_\beta=ann(q)$ for some $q\in Q_\beta$. To see this, it is clear that if $Min(R)$ is finite then $Min(R)$ is compact and for each $\beta$, $Q_\beta\nsubseteq P_{\beta}$. Conversely, assume that $Min(R)$ is compact and for each $\beta$ let $x_{\beta}\in Q_{\beta}\setminus P_{\beta}$. Therefore $V_m(x_{\beta})^c=\{P_\beta\}$. Now, since the collection $\{V_m(x_{\beta})^c\}$ is an open cover for $Min(R)$ and $Min(R)$ is compact, we immediately conclude that $Min(R)$ is finite. Finally, if in addition $R$ is a reduced ring then clearly for each $\beta$, $Q_{\beta}\neq 0$ and $Q_{\beta}\cap P_{\beta}=N(R)=0$. which shows that each $P_{\beta}$ is not essential and $Q_\beta P_{\beta}=0$. Therefore $P_{\beta}\subseteq ann(Q_{\beta})$ and since $Q_\beta\nsubseteq P_\beta$ we obtain that $ann(Q_\beta)=P_\beta$. This immediately shows that for each $q\in Q_\beta\setminus P_\beta$ we have $P=ann(q)$.
\end{rem}

By the above remark we give another proof for finiteness of $Min(R)$ for noetherian rings.

\begin{cor}
Let $R$ be a noetherian ring, then $Min(R)$ is finite.
\end{cor}
\begin{proof}
First note that $Spec(R)$ is noetherain, since $R$ is noetherian. Therefore $Min(R)$ is compact. Without loss of generality we may assume that $R$ is a reduced ring. Let $Min(R)=\{P_\alpha\}$, hence for each $\beta$ we conclude that there exist $s_\beta\in R\setminus P_\alpha$ such that $P_\beta=ann(s_\beta)$. Clearly $s_\beta\in Q_\beta$ and therefore by the above remark we are done.
\end{proof}

In this paper all ring are commutative with $1\neq 0$. All subrings, Modules and ring homomorphisms are unital. A subgroup $S$ of $(R,+)$, which is closed under multiplication of $R$ and $1\notin S$ is called Subring-1. Let $\mathcal{S}$ be a family of (certain) subset of a ring $R$, then we say that a family $\mathcal{I}$ of ideals of $R$ has $A$-property for $\mathcal{S}$, if $S\in\mathcal{S}$ is covered by $\mathcal{I}$, then $S\subseteq I$ for some $I\in\mathcal{I}$. In particular, if $\mathcal{S}$ is the family of all ideals of $R$, and $\mathcal{I}$ has $A$-property for $\mathcal{S}$, then we say that $\mathcal{I}$ has $A$-property (for $R$). One can easily see that if $\mathcal{C}$ is a chain of ideals of $R$, then $\mathcal{C}$ has $A$-property for all finitely generated ideals of $R$. When $\mathcal{I}\subseteq Spec(R)$ and $\mathcal{I}$ has $A$-property for $\mathcal{S}$, we say $\mathcal{I}$ has  $P.A$-property for $\mathcal{S}$. Clearly, each finite set of prime ideals of a ring $R$ has $P.A$-property for the set of all subrings$-1$ of $R$ and therefore has $P.A$-property for $R$.\\

A brief outline of this paper is as follow. In the next section we first proved some basic facts for (certain) $\mathcal{A}\subseteq Spec(R)$, that has (not) $P.A$-property. In particular, we prove that each compact set of primes of a zero-dimensional rings has $P.A$-property. We give a characterization of a set $\mathcal{A}$ of noncomparable prime ideals of a ring $R$ which has $P.A$-property by ring homomorphism; in fact it is shown that  if $\mathcal{A}$ has $P.A$-property, then $\mathcal{A}$ is equal to the inverse image of the set of all maximal ideals of a certain ring $T$ under a ring homomorphism from $R$ into $T$. Next, we determine exactly when $\mathcal{A}\subseteq Spec(R)$ has $P.A$-property. In fact, we prove that $\mathcal{A}$ has $P.A$-property if and only if $\mathcal{A}$ is compact (with Zariski Topology) and $\mathcal{A}$ has $P.A$-property for finitely generated ideals. We show that if $K$ is a formally real field and $X$ is a subset of affine space $K^n$, then $\mathcal{M}_X:=\{M_P\ |\ P\in X\}$ has $P.A$-property. We observe that if $\mathcal{A}\subseteq Spec(R)$, where $R$ is a Bezout ring, zero-dimensional ring or $R=C(X)$ for a topological space $X$, then $\mathcal{A}$ has $P.A$-property if and only if $\mathcal{A}$ is compact. We give an example which shows the previous is not true for Prufer domains, even if $\mathcal{A}\subseteq Spec(R)$ is compact. We show that if $R$ is a $QR$-domain, then $\mathcal{A}\subseteq Spec(R)$ has $P.A$-property for $R$ if and only if $\mathcal{A}$ is compact. If $R$ is a Prufer domain and each $\mathcal{A}$ has $P.A$-property, then we show that $R$ is a $QR$-domain. In particular, if $R$ is a Dedekind domain, then each $\mathcal{A}\subseteq Spec(R)$ has $P.A$-property if and only if $R$ is a $QR$-domain (i.e., $R$ has torsion class group). It is shown that if $f:R\rightarrow C(X)$ is a ring homomorphism from a ring $R$ to $C(X)$ and $\mathcal{A}\subseteq Spec(C(X))$ is compact, then $\mathcal{A}'=\{f^{-1}(P)\ |\ P\in\mathcal{A}\}$ has $P.A$-property for $R$. In particular, if $I$ is an ideal of $R$ then $I^e=C(X)$ if and only if $f(I)$ contains a unit of $C(X)$. We also prove some corollaries for $A$-property of a finite/countable set of ideals in infinite artinian/noetherian rings with certain cardinality, related to the results of \cite{mcadm}, \cite{qb} and \cite{shvm}. In particular, if $R$ is a noetherian integral domain with $|R|>2^{\aleph_0}$, then $R$ is a $u$-ring, i.e., each finite set of ideals of $R$ has $A$-property. Finally, we prove Davis avoidance Theorem for valaution domains instead of prime ideals. In fact we show that if $K$ is a field, $V,V_1,\ldots, V_n$ be valuations for $K$ and $x\in K$ and $V[x]\nsubseteq \bigcup_{i=1}^n V_i$, then there exists $v\in V$ such that $v+x\notin \bigcup_{i=1}^n V_i$.

\section{Main Result}

We begin this section by the following immediate application of Krull Maximal Ideal Theorem.

\begin{prop}\label{np1}
Let $R$ be a ring and $I$ be an ideal of $R$.
\begin{enumerate}
\item $V(I)$ and $V_M(I)$ have $P.A$-property for $R$. In particular, if $M$ is a finitely generated $R$-module then $supp(M)$ has $P.A$-property.
\item Let $\mathcal{A}\subseteq Spec(R)$ has not $P.A$-property for $R$, then there exists a prime ideal $Q\notin \mathcal{A}$ such that $\bigcap_{P\in \mathcal{A}} P \subseteq Q$.
\item Let $\mathcal{A}\subseteq Spec(R)$ and $I$ be an ideal of $R$ such that $I\subseteq \bigcup_{P\in \mathcal{A}} P$. Then either there exists a prime ideal $P$ in $\mathcal{A}$ such that $I\subseteq P$ or $\mathcal{A}\subsetneq V(Ann(I))$. Hence in the latter condition $I+Ann(I)$ is a proper ideal of $R$.
\item Let $\mathcal{A}\subseteq Spec(R)$ and $\mathcal{A}'$ be the set of ideals $I$ of $R$ which are contained in $\bigcup_{P\in \mathcal{A}} P$, but $I\nsubseteq P$ for each $P\in \mathcal{A}$. Then either $\mathcal{A}$ has $P.A$-property or $\mathcal{A}\subsetneq V(I')$ where $I'=\sum_{I\in \mathcal{A}'} Ann(I)$.
\item If $R$ is a zero-dimensional ring and $\mathcal{A}\subseteq Spec(R)$ is compact. Then $\mathcal{A}$ has $P.A$-property for $R$.
\end{enumerate}
\end{prop}
\begin{proof}
The first part of $(1)$ is clear by Krull Maximal Ideal Theorem. For the final part of $(1)$ note that since $M$ is finitely generated we conclude that $supp(M)=V(Ann(M))$ and therefore by the first part we are done. For $(2)$, since $\mathcal{A}$ has not $P.A$-property, we infer that $\mathcal{A}\subsetneq V(\bigcap_{P\in\mathcal{A}}P)$, by $(1)$. To see $(3)$ assume that for each $P\in\mathcal{A}$, $I\nsubseteq P$. Therefore $Ann(I)\subseteq P$, for each $P\in\mathcal{A}$. Thus $\mathcal{A}\subseteq V(Ann(I))$. Hence by $(1)$ we are done for $(3)$. $(4)$ is clear by $(3)$. Finally for $(5)$, note that in this case $Spec(R)$ is a Hausdorff space, therefore $\mathcal{A}$ is close. Thus by $(1)$, $\mathcal{A}$ has $P.A$-property.
\end{proof}

Let $R$ be a ring a subset $S$ of $R$ is called a multiplicatively closed set if $0\notin X$, $1\in X$ and $X$ is closed under multiplication of $R$. A well known theorem of I.S. Cohen (which is a generalization of Krull Maximal Ideal Theorem) shows that an ideal $I$ of a ring $R$ disjoint from a multiplicatively closed set $X$ of $R$ if and only if there exists a prime ideal $P$ of $R$ which contains $I$ and disjoint from $X$, see \cite[Theorem 1]{kap}. The proof of the following proposition which is an immediate consequences of Cohen Theorem and the structure of (prime) ideals of ring of quotient of $R$ respect to a multiplicatively close sets (see \cite[Sec. 1-4]{kap}) is simple and hence left to the reader.

\begin{prop}\label{p1}
Let $R$ be a ring and $\mathcal{A}\subseteq Spec(R)$. The following are equivalent:
\begin{enumerate}
\item $\mathcal{A}$ has $P.A$-property for $R$.
\item $\mathcal{A}$ has $P.A$-property for $Spec(R)$.
\item $Max(R_X)\subseteq \{P_X\ |\ P\in \mathcal{A}\}$, where $X=R\setminus (\bigcup_{P\in \mathcal{A}}P)$.
\end{enumerate}
\end{prop}

We remind that in fact $Q(R)$ is VNR for a reduced ring $R$ if and only if $Max(Q(R))=\{P_X\ |\ P\in Min(R)\}$, where $X=R\setminus \bigcup_{P\in Min(R)} P$ is the set of regular (nonzero divisors) of $R$. In the following theorem we generalize the previous fact for arbitrary set of incomparable prime ideals of a ring $R$ and show that the $P.A$-property is in fact the Krull Maximal Ideal Theorem.

\begin{thm}
\begin{enumerate}
\item Let $R$ be a ring and $\mathcal{A}\subseteq Spec(R)$ be an incomparable set of primes which has $P.A$-property for $R$. Then there exists a ring $T$ and a ring homomorphism $f:R\rightarrow T$ such that $\mathcal{A}=\{f^{-1}(M)\ |\ M\in Max(T)\}$.
\item Let $R$ and $T$ be rings and $f:R\rightarrow T$ be a ring homomorphism. Then $\mathcal{A}=\{f^{-1}(M)\ |\ M\in Max(T)\}$ has $P.A$-property for $R$ if and only if for each ideal $I$ of $R$ with $I\cap X=\emptyset$, the ideal $I^e$ is proper in $T$, where $X=R\setminus \bigcup_{P\in\mathcal{A}} P$. Moreover in this case $f$ can be extended to a ring homomorphism from $R_X$ to $T$.
\end{enumerate}
\end{thm}
\begin{proof}
$(1)$ It suffices to put $T=R_X$, where $X=R\setminus \bigcup_{P\in\mathcal{A}} P$ and $f$ be the natural ring homomorphism from $R$ into $T$. Now note that since $\mathcal{A}$
has $P.A$-property and elements of $\mathcal{A}$ are incomparable, we immediately infer that $Max(T)=\{P_X\ |\ P\in \mathcal{A}\}$, by the structure of prime ideals of $T=R_X$. It is clear that for each $P\in\mathcal{A}$ we have $f^{-1}(P_X)=P$ which complete the proof of $(1)$.\\
$(2)$ For the if part, let $I$ be an ideal of $R$ which is contained in the union of element of $\mathcal{A}$. Then $I\cap X$ is empty and therefore $I^e$ is a proper ideal of $T$. Hence by Krull Maximal Ideal Theorem $I^e$ is contained in a maximal ideal $M$ of $T$. Thus $I\subseteq f^{-1}(M)\in\mathcal{A}$. Conversely, assume that $\mathcal{A}$ has $P.A$-property for $R$ and $I$ be an ideal of $R$ with $I\cap X$ is empty. Thus $I$ is contained in the union of element of $\mathcal{A}$ and therefore by assumption $I$ is contained in $f^{-1}(M)$ for some $M\in Max(T)$. Hence $I^e$ is contained in $M$ and therefore $I^e$ is proper ideal. Finally, for each $t\in X$, we have $f(t)\notin \bigcup_{M\in Max(T)} M$ and therefore $f(t)\in U(T)$. This immediately shows that $f$ can be extended to a ring homomorphism from $R_X$ to $T$.
\end{proof}

Now we have the following result.

\begin{prop}\label{p2}
Let $f$ be a ring homomorphism from a ring $R$ to a ring $T$ and $\mathcal{A}$ be a set of ideals of $T$ which has $A$-property for subrings$-1$ of $T$. Then $\mathcal{A}':=\{Q^c=f^{-1}(Q)\ |\ Q\in \mathcal{A}\}$ has $A$-property for subrings$-1$ of $R$.
\end{prop}
\begin{proof}
Assume that $S$ be a subring$-1$ of $R$ which is contained in $\bigcup_{Q\in \mathcal{A}} Q^c$. Therefore we infer that $f(S)$ is contained in $\bigcup_{Q\in \mathcal{A}}Q$. Now note that $f(S)$ is a subring$-1$ of $T$. Hence by our assumption we infer that there exists a $Q\in \mathcal{A}$ such that $f(S)\subseteq Q$. Therefore $S\subseteq Q^c$ and we are done.
\end{proof}

In the next result, we determine exactly when $\mathcal{A}\subseteq Spec(R)$ has $P.A$-property.

\begin{thm}\label{p3}
Let $R$ be a ring and $\mathcal{A}\subseteq Spec(R)$. The following conditions are equivalent:
\begin{enumerate}
\item $\mathcal{A}$ is compact and has $P.A$-property for all finitely generated ideals of $R$.
\item $\mathcal{A}$ has $P.A$-property for $R$.
\end{enumerate}
\end{thm}
\begin{proof}
$(1)\Rightarrow (2)$: Let $I$ be an ideal of $R$ which is contained in the union of $\mathcal{A}=\{P_{\alpha}\}_{\alpha\in \Gamma}$ but for each $\alpha\in\Gamma$, $I$ is not contained in $P_{\alpha}$. Therefore for each $\alpha\in\Gamma$, there exists $x_{\alpha}\in I\setminus P_{\alpha}$. Hence $P_{\alpha}\in V_{\mathcal{A}}(x_{\alpha})^c$ for each $\alpha\in\Gamma$. Thus the collection $\{V_{\mathcal{A}}(x_{\alpha})^c\}_{\alpha\in\Gamma}$ is an open cover for $\mathcal{A}$. Now by our assumption $\mathcal{A}$ is compact, and therefore there exist finitely many $\alpha_1,\ldots,\alpha_n$ in $\Gamma$ such that $\mathcal{A}=V_{\mathcal{A}}(x_{\alpha_1})^c\cup\cdots\cup V_{\mathcal{A}}(x_{\alpha_n})^c$. Thus $V_{\mathcal{A}}(x_{\alpha_1},\ldots,x_{\alpha_n})=\emptyset$. But the finitely generated ideal $J=<x_{\alpha_1},\ldots,x_{\alpha_n}>$ of $R$ is contained in $I$ and therefore in the union of $\mathcal{A}$, which by $(1)$ immediately implies that there exists $\alpha\in\Gamma$ such that $J\subseteq P_{\alpha}$, i.e., $P_{\alpha}\in V_{\mathcal{A}}(x_{\alpha_1},\ldots,x_{\alpha_n})$, which is a contradiction.\\
$(2)\Rightarrow (1)$: It suffices to show that $\mathcal{A}$ is compact. Hence assume that $\mathcal{A}=\bigcup_{\alpha\in\Gamma}V_{\mathcal{A}}(I_{\alpha})^c$, where each $I_\alpha$ is an ideal of $R$. Therefore $V_{\mathcal{A}}(\sum_{\alpha\in\Gamma} I_{\alpha})=\emptyset$. Thus by $(2)$ we conclude that $\sum_{\alpha\in\Gamma} I_{\alpha}\nsubseteq \bigcup_{P\in \mathcal{A}}P$. This immediately implies that there exist finitely many $\alpha_1,\ldots,\alpha_n$ in $\Gamma$ such that $I_{\alpha_1}+\cdots +I_{\alpha_n}\nsubseteq \bigcup_{P\in \mathcal{A}}P$. Therefore $V_{\mathcal{A}}(I_{\alpha_1}+\cdots +I_{\alpha_n})=\emptyset$, i.e., $\mathcal{A}$ is compact and hence we are done.
\end{proof}

Now we give some conclusions of the above theorem.

\begin{cor}\label{p4}
Let $R$ be a ring and $\mathcal{A}\subseteq Spec(R)$ be a chain. Then $\mathcal{A}$ is compact if and only if $\mathcal{A}$ has $P.A$-property. In particular, if $\mathcal{A}$ is compact then $\bigcup_{P\in \mathcal{A}} P\in \mathcal{A}$.
\end{cor}

\begin{exm}
\begin{enumerate}
\item let $K$ be a field and $R=K[x_1,x_2,\ldots]$ be the ring of polynomials of independent variables $x_1,x_2,\ldots$ over $K$. Then clearly $I=(x_1,x_2,\ldots)$ is
   contained in the union of primes ideals $P_n=(x_1,x_2,\ldots,x_n)$, but $I$ is not contained in $P_n$ for each $n$.
\item Let $V$ be a valuation ring, $\mathcal{A}$ a set of prime ideals of $V$ and $I$ an ideal of $V$ which is contained in $\bigcup \mathcal{A}$ but is not contained in any element of $\mathcal{A}$, then $I=\bigcup\mathcal{A}$ and therefore $I$ is prime.
\end{enumerate}
\end{exm}

The following is similar to \cite[Proposition 2.5]{shvm}.

\begin{cor}
Let $R$ be a ring such that there exists an uncountable family $\{t_{\alpha}\}_{\alpha\in\Gamma}$
of elements of $R$ such that for each $\alpha\neq \beta$ in $\Gamma$ we have $t_{\alpha}-t_{\beta}\in U(R)$.
If $\mathcal{A}$ is a countable subset of $Spec(R)$, then $\mathcal{A}$ has $P.A$-property for $R$ if and only $\mathcal{A}$ is compact.
\end{cor}
\begin{proof}
By \cite[Proposition 2.5]{shvm}, $\mathcal{A}$ has $P.A$-property for finitely generated ideals of $R$. Hence by Theorem \ref{p3}, we infer that $\mathcal{A}$
has $P.A$-property for $R$ if and only if $\mathcal{A}$ is compact.
\end{proof}

\begin{cor}
Let $R$ be a zero-dimensional ring (in particular, if $R$ is VNR) and $\mathcal{A}\subseteq Spec(R)$. Then $\mathcal{A}$ has $P.A$-property if and only if $\mathcal{A}$ is compact (closed, i.e., $\mathcal{A}=V(I)$ for some ideal $I$ of $R$).
\end{cor}
\begin{proof}
The if part is evident by $(5)$ of Proposition \ref{np1} and the converse holds by Theorem \ref{p3} (and the fact that $Spec(R)$ is Hausdorff for zero-dimensional rings).
\end{proof}

\begin{cor}
Let $R$ be a Bezout ring (i.e., every finitely generated ideal of $R$ is principal) and $\mathcal{A}\subseteq Spec(R)$.
Then $\mathcal{A}$ has $P.A$-property if and only if $\mathcal{A}$ is compact.
\end{cor}
\begin{proof}
Assume that $\mathcal{A}$ be a set of ideals of $R$ (not necessary prime) and $I$ be a finitely generated ideal of $R$ which is contained in $\bigcup\mathcal{A}$. Since $R$ is a Bezout ring we infer that $I=Ra$ for some $a\in R$. Thus there exists $J$ in $\mathcal{A}$ such that $a\in J$ and therefore $I\subseteq J$. Hence we are done by Theorem \ref{p3}.
\end{proof}

\begin{rem}
The above corollary is still true if $R$ is an almost Bezout domain, see \cite{andzaf}. The proof is similar and need to use \cite[Lemma 3.4]{andzaf}.
\end{rem}

In the next theorem we show that the above corollary does not hold for Prufer domains. We remind the reader that if $R$ is a Prufer domain with quotient field $K$ and $T$ be an overring of $R$, then each prime ideal $Q$ of $T$ has the form $PT$ where $P=Q\cap T$ and in this case $R_P=T_Q$, see \cite[Theorem 26.1]{gilb}. Also we remind the reader that an integral domain $D$ is called $QR$-domain if each overring of $D$ is a quotient of $D$ respect to a multiplicatively closed set of $D$. Finally, note that each $QR$-domain is Prufer, see $\S 27$ of \cite{gilb}.

\begin{thm}
\begin{enumerate}
\item Let $R$ be a Prufer domain with quotient field $K$, $T$ be an overring of $R$ and $\mathcal{A}=\{P\in Spec(R)\ |\ PT\in Max(T) \}$. Then $\mathcal{A}$ is compact.
\item If $R$ is a Prufer domain which is not a $QR$-domain, then there exists a compact set of prime ideals in $R$ which has not $P.A$-property.
\item If $R$ is a $QR$-domain and $\mathcal{A}\subseteq Spec(R)$, then $\mathcal{A}$ has $P.A$-property if and only if $\mathcal{A}$ is compact.
\end{enumerate}
\end{thm}
\begin{proof}
For $(1)$, let $\mathcal{A}=\bigcup_{\alpha\in\Gamma}V_{\mathcal{A}}(I_\alpha)^c$, where each $I_{\alpha}$ is an ideal of $R$. Hence we infer that $V_{\mathcal{A}}(\sum_{\alpha\in\Gamma} I_{\alpha})=\emptyset$. Now since each maximal ideals of $T$ has the form $PT$ for some $P\in\mathcal{A}$, we conclude that
$(\sum_{\alpha\in\Gamma} I_{\alpha})T=T$. Therefore there exists $\alpha_1,\ldots,\alpha_n$ in $\Gamma$ such that $(I_{\alpha_1}+\cdots+I_{\alpha_n})T=T$. Thus $V_{\mathcal{A}}(I_{\alpha_1}+\cdots+I_{\alpha_n})=\emptyset$ which immediately implies that $\mathcal{A}$ is compact.\\
$(2)$ Assume that $T$ be an overring of $R$ which is not a quotient of $R$ respect to multilplicatively closed subsets of $R$. Suppose $Max(T)=\{ Q=PT\ |\ P\in\mathcal{A}\}$, where $\mathcal{A}\subseteq Spec(R)$ and $X=R\setminus\bigcup_{P\in\mathcal{A}}P$. Hence by the above comment we conclude that

$$R_X\subsetneq T=\bigcap_{Q\in Max(T)} T_Q=\bigcap_{P\in\mathcal{A}} R_P$$

Now if $\mathcal{A}$ has $P.A$-property then by \cite[Proposition 4.8]{gilb}, $T=R_X$ which is absurd. Hence $\mathcal{A}$ is compact by $(1)$ and has not $P.A$-property.\\
Finally for $(3)$, we first remind that if $R$ is a $QR$-domain then for each finitely generated ideal $I$ of $R$, there exists a natural number $n$ and $a\in I$ such that $I^n\subseteq (a)$, see \cite[Theorem 27.5]{gilb}. Now assume that $I$ is a finitely generated ideal of $R$ which is contained in $\bigcup \mathcal{A}$, then by the latter fact we immediately infer that $I$ is contained in an element of $\mathcal{A}$ and therefore by Theorem \ref{p3} we are done.
\end{proof}

\begin{rem}
Let $R$ be a ring. If all subset of $Spec(R)$ has $P.A$-property, then $Spec(R)$ is a noetherian space, by Theorem \ref{p3}. In \cite{smit}, Smith proved a more stronger result: all subset of $Spec(R)$ has $P.A$-property if and only if for each prime ideal $P$ of $R$ there exist $x$ such that $P=\sqrt{(x)}$. As Gilmer mentioned in \cite[Proposition 4]{gilint}, the latter fact is equivalent to: for each ideal $I$ of $R$ there exist $a\in I$ such that $\sqrt{I}=\sqrt{(a)}$ (which immediately implies that $Spec(R)$ is noetherian).
\end{rem}

\begin{cor}\label{nqrd}
\begin{enumerate}
\item Let $R$ be a noetherian $QR$-domain. Then each subset of $Spec(R)$ has $P.A$-property.
\item Let $R$ be a Prufer domain. If each $\mathcal{A}\subseteq Spec(R)$ has $P.A$-property, then $R$ is a $QR$-domain.
\item If $R$ is a Dedekind domain, then each subset of $Spec(R)$ has $P.A$-property if and only if $R$ is a $QR$-domain (if and only if $R$ has torsion class group).
\end{enumerate}
\end{cor}
\begin{proof}
First note that since $R$ is noetherian then each subset of $Spec(R)$ is compact. Thus if $R$ is a $QR$-domain, then by $(3)$ of the previous theorem each subset of $Spec(R)$ has $P.A$-property. Thus $(1)$ holds. For $(2)$, assume that each subset of $Spec(R)$ has $P.A$-property, then by  \cite[Proposition 4]{gilint} (or \cite{smit}), for each ideal $I$ of $R$ there exists $a\in I$ such that $\sqrt{I}=\sqrt{(a)}$. Hence if $I$ is a finitely generated ideal of $R$ we infer that there exists a natural number $n$ such that $I^n\subseteq (a)$ and therefore by \cite[Theorem 27.5]{gilb}, we conclude that $R$ is a $QR$-domain. $(3)$ is evident by $(1)$ and $(2)$. For the parenthesis fact of $(3)$ see \cite[Theorem 40.3]{gilb}.
\end{proof}

\begin{rem}
We remind the reader that the following are equivalent for an integral domain $R$,
\begin{enumerate}
\item $R$ is a $PID$.
\item $R$ is noetherian and each maximal ideal of $R$ is principal.
\item $R$ has $ACC$ on principal ideals and each maximal ideal of $R$ is principal.
\item $R$ is atomic and each maximal ideal of $R$ is principal.
\end{enumerate}
To see this first note that clearly $(1)\Leftrightarrow (2)\Rightarrow (3)\Rightarrow (4)$. Hence it remains to show that $(4)$ implies $(1)$. We may assume that $R$ is not a field. Let $M=(p)$ be an arbitrary maximal ideal of $R$, thus $p$ is a prime element of $R$. We claim that $J:=\bigcap_{n=1}^\infty (p^n)=0$. Suppose that $J\neq 0$, thus by
\cite[Exersice 5, P. 7]{kap}, we infer that $J$ is a prime ideal and $J\subsetneq M$. Since $J$ is a nonzero prime ideal of $R$ and $R$ is atomic, we conclude that there exists an irreducible element $q\in J$. Thus $q\in M=(p)$ and therefore $q=p$, which is absurd. Hence $J=0$. Again by \cite[Exersice 5, P. 7]{kap}, we deduce that $M$ contains no properly nonzero prime ideal (we refer the reader to \cite{andpp} for more interesting results about principal prime ideals in any commutative ring). Hence $R$ is a $PID$.
\end{rem}

\begin{cor}
Let $R$ be an integral domain. Then $R$ is a $PID$ if and only if $R$ is a $UFD$ and the family of all principal prime ideals of $R$ has $P.A$-property.
\end{cor}
\begin{proof}
It is clear that if $R$ is a $PID$, then each subset of (prime) ideals of $R$ has $A$-property. Conversely, assume that $R$ is a $UFD$ which the family of all principal prime ideals of $R$ has $P.A$-property. Let $Ir(R)$ be the set of all prime elements of $R$ up to associate. Then clearly $M\subseteq \bigcup_{p\in Ir(R)}(p)$, since $R$ is a $UFD$. Therefore by assumption we conclude that $M\subseteq (p)$, for some $p\in Ir(R)$, therefore $M$ is principal. Thus $R$ is a PID.
\end{proof}

We remind the reader that if $X$ is a completely regular Hausdorff topological space, then $C(X)$ denotes the ring of all continuous real functions on $X$.
For each $x\in X$, let $M_x=\{f\in C(X)\ |\ f(x)=0\}$, it is clear that $C(X)/M_x\cong \mathbb{R}$ as ring and therefore $M_x\in Max(C(X))$, for each $x\in X$.

\begin{cor}\label{t5}
 Let $X$ be a topological space and $\mathcal{A}\subseteq Spec(C(X))$. Then $\mathcal{A}$ has $P.A$-property for finitely generated ideals of $R$. consequently, $\mathcal{A}$ has $P.A$-property if and only if $\mathcal{A}$ is compact. In particular, if $A$ is a subset of $X$ satisfies at least one of the following conditions:
\begin{enumerate}
\item $X$ is compact and $A$ is a closed subset of $X$.
\item $A$ is a compact subset of $X$.
\end{enumerate}
Then $\mathcal{M}_A=\{M_a\ |\ a\in A\}$ has $P.A$-property for $C(X)$.
\end{cor}
\begin{proof}
First we remind that for each prime ideal $P$ of $C(X)$, the ring $C(X)/P$ is a totally ordered ring, see \cite[Theorem 5.5]{gljr}. Now let $I=<f_1,\ldots,f_n>$ be a finitely generated ideal of $C(X)$ which is contained in the union of $\mathcal{A}$. Thus there exists a $P\in \mathcal{A}$ such that $f_1^2+\cdots+f_n^2\in P$, which by the previous fact immediately implies that each $f_i\in P$. Therefore $I\subseteq P$. Thus by Theorem \ref{p3}, we conclude that $\mathcal{A}$ has $P.A$-property if and only if $\mathcal{A}$ is compact. Now assume that $(1)$ or $(2)$ holds, then it is obvious that $A$ is compact in case $(1)$. Now one can easily see that $\mathcal{M}_A$ is compact and therefore by the previous part has $P.A$-property.
\end{proof}

\begin{cor}
Let $R$ be a ring, $X$ a topological space and $f:R\rightarrow C(X)$ a ring homomorphism. Then the following hold:
\begin{enumerate}
\item if $\mathcal{A}\subseteq Spec(C(X))$ is compact, then $\mathcal{A}'=\{f^{-1}(P)\ |\ P\in\mathcal{A}\}$ has $P.A$-property for $R$.
\item if $I$ is an ideal of $R$ then $I^e=C(X)$ if and only if $f(I)$ contains a unit of $C(X)$.
\end{enumerate}
\end{cor}
\begin{proof}
Let $I$ be an ideal of $R$ which is contained in $\bigcup_{P\in \mathcal{A}'}f^{-1}(P)$. Thus $f(I)$ is contained in $S:=\bigcup_{P\in \mathcal{A}} P$.
We claim that $I^e=f(I)C(X)$ is contained in $S$ and therefore is a proper ideal of $C(X)$. If $a_1,\ldots,a_n$ are in $I$, then $f(a_1)^2+\cdots+f(a_n)^2\in f(I)\subseteq S$, thus we conclude that there exists $P\in \mathcal{A}$ such that $f(a_1)^2+\cdots+f(a_n)^2\in P$. Since $C(X)/P$ is a totally ordered ring we infer that $f(a_i)\in P$, for each $i$. Therefore we conclude that $C(X)f(a_1)+\cdots+C(X)f(a_n)\subseteq P$, which shows $I^e\subseteq S$ and hence $I^e$ is proper. Now since $\mathcal{A}$ has $P.A$-property, we deduce that there exists $Q\in\mathcal{A}$ such that $I^e\subseteq Q$ and therefore $I\subseteq f^{-1}(Q)\in\mathcal{A}'$ and hence $(1)$ holds. For $(2)$, if $I^e=C(X)$ but $f(I)$ contains no unit of $C(X)$, then we conclude that $f(I)\subseteq \bigcup_{M\in Max(C(X))}M$. Therefore $I\subseteq \bigcup_{M\in Max(C(X))}f^{-1}(M)$. Thus by part $(1)$, we infer that there exist $M\in Max(C(X))$ such that $I\subseteq f^{-1}(M)$, i.e., $I^e\subseteq M$ which is absurd. Therefore $f(I)$ contains a unit of $C(X)$.
\end{proof}

Let $K$ be a field, $P$ be a point in the affine space $K^n$, and $R=K[x_1,\ldots,x_n]$ be polynomial ring of $n$ variable over $K$. Then $M_P=\{f\in R\ |\ f(P)=0\}$ is a maximal ideal of $R$. We remind that $Spec(R)$ and $K^n$ are noetherian spaces (by Zariski topologies). Now the following is in order.

\begin{cor}
Assume that $K$ be a formally real field and $X$ a subset of affine space $K^n$. Then $\mathcal{M}_X=\{M_P\ |\ P\in X\ \}$ has $P.A$-property for $R=K[x_1,\ldots,x_n]$.
\end{cor}
\begin{proof}
Let $I=<f_1,\ldots,f_n>$ be an ideal of $R$ which is contained in $\bigcup_{P\in X}M_P$. Now since $f:=f_1^2+\cdots+f_n^2\in I$, we infer that there exists $P\in X$ such that $f\in M_P$. Thus $f(P)=0$ and since $K$ is formally real we immediately conclude that $f_i(P)=0$ for each $i$. Thus $I\subseteq M_P$ and we are done.
\end{proof}

\begin{thm}
Let $R$, $T$ be rings and $f$  a ring homomorphism from $R$ into $T$. If $\mathcal{A}\subseteq Spec(T)$ has $P.A$-property, then $\mathcal{A}'=\{f^{-1}(P)\ |\ P\in\mathcal{A}\}$ is compact. Moreover, if $Im(f)\leq T$ has lying-over (in particular, if $T$ is integral over $Im(f)$), then $\mathcal{M}=\{f^{-1}(M)\ |\ M\in Max(T)\}$ has $P.A$-property for $R$.
\end{thm}
\begin{proof}
Let $\{V_{\mathcal{A}'}(I_{\alpha})^c\}_{\alpha\in\Gamma}$ be an open cover for $\mathcal{A}'$, where each $I_{\alpha}$ is an ideal of $R$. Hence we infer that $V_{\mathcal{A}'}(I)=\emptyset$, where $I=\sum_{\alpha\in\Gamma}I_{\alpha}$. Thus for each $P\in\mathcal{A}$ we have $I\nsubseteq f^{-1}(P)$. Therefore for each $P\in\mathcal{A}$ we conclude that $I^e\nsubseteq P$. Now since $\mathcal{A}$ has $P.A$-property we infer that $I^e\nsubseteq \bigcup_{P\in\mathcal{A}}P$. Hence we deduce that there exist $\alpha_1,\ldots,\alpha_n$ in $\Gamma$ such that $(I_{\alpha_1}+\cdots+I_{\alpha_n})^e\nsubseteq \bigcup_{P\in\mathcal{A}}P$. Thus for each $P\in\mathcal{A}$ we have $(I_{\alpha_1}+\cdots+I_{\alpha_n})^e\nsubseteq P$ and therefore $I_{\alpha_1}+\cdots+I_{\alpha_n}\nsubseteq f^{-1}(P)$ for each $P\in\mathcal{A}$. This shows that $\mathcal{A}'=V_{\mathcal{A}'}(I_{\alpha_1})^c\cup\cdots\cup V_{\mathcal{A}'}(I_{\alpha_n})^c$, i.e., $\mathcal{A}'$ is compact.\\
Now suppose that $Im(f)\leq T$ has lying-over and $I$ be an ideal of $R$ which is contained in union of $\mathcal{M}$. Thus $f(I)\subseteq \bigcup_{N\in Max(Im(f))} N$. Therefore by Krull Maximal Ideal Theorem, $f(I)$ is contained in a maximal ideal $N$ of $Im(f)$. Since lying-over holds, we conclude that there exists a maximal ideal $M$ of $T$ over $N$ and therefore contains $f(I)$. Thus $I$ is contained in $f^{-1}(M)$ and we are done.
\end{proof}

We remind that if $M$ is a finitely generated $R$-module, then $supp(M)=V(ann(M))$.

\begin{prop}
Let $R$ be a ring and $\mathcal{A}=\{P_{\alpha}\}_{\alpha\in\Gamma}$ be a compact set of prime ideals of $R$. If $I$ is an ideal of $R$ which is contained in $\bigcup_{\alpha\in \Gamma}P_{\alpha}$, then either $\mathcal{A}\subseteq supp(I)$ or $I\subseteq P_{\alpha}$ for some $\alpha\in\Gamma$. In particular, if $I$ is finitely generate and for each $\alpha\in\Gamma$, $I\nsubseteq P_{\alpha}$, then $\mathcal{A}\subsetneq supp(I)$.
\end{prop}
\begin{proof}
Similar to the proof of Theorem \ref{p3}, if for each $\alpha$, $I$ is not contained in $P_{\alpha}$, then there exists a finitely generated ideal $J$ of $R$ which is contained in $I$, but $J$ is not contained in each $P_{\alpha}$. Hence we infer that for each $\alpha\in \Gamma$, $Ann(J)\subseteq P_{\alpha}$. Thus by the above comments for each $\alpha\in\Gamma$, $P_{\alpha}\in supp(J)\subseteq supp(I)$. The final part is evident by Proposition \ref{np1}.
\end{proof}

Let us remind the reader some needed facts from the liturature for the next results. In \cite{mcadm}, McAdam proved that if $R$ is a ring such that for each maximal ideal $M$ of $R$, the residue ring $R/M$ is infinite, then each finite set of ideals of $R$, has $A$-property, i.e., if $I, J_1,\ldots,J_n$ are ideals of $R$ and $I\subseteq \bigcup_{k=1}^n J_k$, then $I\subseteq J_k$ for some $k$. In \cite{qb}, Quartararo and Butts, called an ideal $I$ with the latter property a $u$-ideal. They proved that each invertible ideal of a ring $R$ is a $u$-ideal (see \cite[Theorem 1.5]{qb}) and characterized rings for which each ideal of $R$ is a $u$-ideal, and called them $u$-rings. In fact they proved that a ring $R$ is a $u$-ring if and only if for each maximal ideal $M$ of $R$ either $R/M$ is infinite or $R_M$ is a Bezout ring (see \cite[Theorem 2.6]{qb}). Finally, In \cite[Corollary 2.6]{shvm}, Sharpe and Vamos proved that if $(R,M)$ is a local noetherian ring with uncountable residue field and $I,J_1,J_2,\ldots$ are ideals of $R$ such that $I\subseteq \bigcup_{k=1}^\infty J_k$, then $I\subseteq J_k$ for some $k$.

\begin{cor}
Let $(R,M)$ be a local ring which is either an uncountable artinian or a noetherian with $|R|>2^{\aleph_0}$. If $I,J_1,J_2,\ldots$ be ideals of $R$ such that $I\subseteq \bigcup_{n=1}^\infty J_n$, then $I\subseteq I_n$ for some $n\geq 1$.
\end{cor}
\begin{proof}
If $R$ is an uncountable artinian ring, then by the proof of \cite[Proposition 1.4]{azkart}, $R/M$ is uncountable and therefore we are done by \cite[Corollary 2.6]{shvm}. If $R$ is a noetherian with $|R|>2^{\aleph_0}$, then by the proof of \cite[Corollary 2.6]{azkcm}, there exist a natural number $n$ such that $R/M^n$ is uncountable and similar to the proof of the first part $R/M$ is uncountable and hence we are done.
\end{proof}

\begin{cor}
Let $R$ be a noetherian integral domain with $|R|>2^{\aleph_0}$, then $R$ is a $u$-ring.
\end{cor}
\begin{proof}
For each maximal ideal $M$ of $R$, by the proof of \cite[Corollary 2.7]{azkcm}, there exists a natural number $n$ such that the ring $R/M^n$ is uncountable. Therefore similar to the proof of the first part of the previous corollary we infer that $R/M$ is uncountable. Thus we are done by \cite[Theorem 2.6]{qb}.
\end{proof}

Let $R$ be a ring. A proper subring $S$ ($1_R\in S$) is called a maximal subring if there exists no other subring of $R$ between $S$ and $R$. We refer the reader to $[4-7]$ for the existence of maximal subrings in commutative rings.

\begin{cor}
Let $R$ be noetherian integral domain with nonzero characteristic which is not equal to its prime subring (i.e., $R\neq \mathbb{Z}_p$, where $p=Char(R)$) then either $R$ has a maximal subring or $R$ is a $u$-ring.
\end{cor}
\begin{proof}
Assume that $R$ has no maximal subring, then clearly $R$ is infinite and by \cite[Corollary 2.4]{aznot}, we infer that $R$ is countable. Now by \cite[Proposition 3.14]{azkmc}, we conclude that for each proper (maximal) ideal $I$ of $R$, the residue ring $R/I$ is infinite. Thus we are done by \cite[Theorem 2.6]{qb}.
\end{proof}

\begin{cor}
Each infinite artinian local ring is a $u$-ring.
\end{cor}
\begin{proof}
Let $(R,M)$ be an infinite artinian local ring, then by \cite[Corollary 1.5]{aznot}, we deduce that $|R/M|=|R|$. Hence we are done by \cite[Theorem 2.6]{qb}.
\end{proof}

One of the application of prime avoidance lemma is a theorem which referred in \cite[Theorem 124]{kap} to E. Davis: Let $R$ be a commutative ring, $I$  an ideal of $R$, $a\in R$ and $P_1,\ldots, P_n$ prime ideals of $R$; if $Ra+I\nsubseteq \bigcup_{i=1}^n P_i$, then $a+c\notin \bigcup_{i=1}^n P_i$, for some $c\in I$. In particular, if $R$ is a semilocal ring and $Ra+I=R$, then there exists $c\in I$ such that $a+c$ is a unit of $R$. The latter result is true for non-commutative semilocal rings as Bass' Stable Range Theorem. Now the following is in order.

\begin{thm}\label{dvthm}
Let $R$ be a ring, $\mathcal{A}\subseteq Spec(R)$, $I$ an ideal of $R$, $a\in R$ and $Q_1,\ldots Q_n$ be prime ideals of $R$. Assume that $Ra+I \nsubseteq (\bigcup_{P\in V_{\mathcal{A}}(a)}P)\cup Q_1\cup\cdots\cup Q_n$. If $V_{\mathcal{A}}(a)$ has $P.A$-property for $R$ (in particular, if $V(a)\subseteq \mathcal{A}$ or $\mathcal{A}\cap Spec(R)=V_M(a)$), then there exists $c\in I$ such that $a+c\notin (\bigcup_{P\in V_{\mathcal{A}}(a)}P)\cup Q_1\cup\cdots\cup Q_n$.
\end{thm}
\begin{proof}
We may assume that $Q_i\nsubseteq P$, for each $i$ and $P\in V_{\mathcal{A}}(a)$. If $n=0$, then by assumption there exist $c\in I$ and $r\in R$ such that $ra+c\notin\bigcup_{P\in V_{\mathcal{A}}(a)} P$. Since each $P\in V_{\mathcal{A}}(a)$ contains $a$, we infer that $c\notin \bigcup_{P\in V_{\mathcal{A}}(a)} P$ and hence $a+c\notin \bigcup_{P\in V_{\mathcal{A}}(a)} P$. Thus theorem holds for $n=0$. Hence assume that $n\geq 1$. Since $V_{\mathcal{A}}(a)$ has $P.A$-property, we conclude that $Q_1\cap\cdots\cap Q_n\nsubseteq \bigcup_{P\in V_{\mathcal{A}}(a)}P$. Therefore there exists $d\in (Q_1\cap\cdots\cap Q_n)\setminus (\bigcup_{P\in V_{\mathcal{A}}(a)}P)$. Now by assumption there exists $r\in R$ and $b\in I$ such that $ra+b\notin (\bigcup_{P\in V_{\mathcal{A}}(a)}P)\cup Q_1\cup\cdots\cup Q_n$. Now one can easily see that $a+c\not \in (\bigcup_{P\in v(a)}P)\cup Q_1\cup\cdots\cup Q_n$, where $c:=db\in I$ and hence we are done.
\end{proof}

\begin{cor}
Let $R$ be a noetherian $QR$-domain. Assume that $\mathcal{A}\subseteq Spec(R)$, $I$ an ideal of $R$ and $a\in R$ such that $Ra+I\nsubseteq \bigcup_{P\in\mathcal{A}} P$. If $\mathcal{A}\setminus V_{\mathcal{A}}(a)$ is finite, then there exists $c\in I$ such that $a+c\notin \bigcup_{P\in\mathcal{A}} P$.
\end{cor}
\begin{proof}
By Corollary \ref{nqrd}, $V_{\mathcal{A}}(a)$ has $P.A$-property and hence we are done by the previous theorem.
\end{proof}

Theorem \ref{dvthm}, might mislead us to generalize the Bass' Stable Range Theorem for semilocal commutaive rings, to comutative rings in which every non-unit element is not contained in only finitely many maximal ideals, but in fact the latter result is the same result, since one can easily see that (by Zariski topology) a ring $R$ is semilocal if and only if each non-unit element is not contained in only finitely many maximal ideals.\\

Finally in this paper we want to give a valuation version of Davis Theorem. First we need some observation from \cite{azsp}. Let $V, V_1,\ldots, V_n$ be valuations for a field $K$, if $V\subseteq \bigcup_{i=1}^n V_i$, then $V\subseteq V_i$ for some $i$, see \cite[Corollary 3.10]{azsp} (this fact is called Valuation Avoidance Lemma). More generally, if $W_1,\ldots, W_m$ are also valuation for $K$ and $\bigcap_{i=1}^m W_i\subseteq \bigcup_{i=1}^n V_i$, then there exist $i$ and $j$ such that $W_j\subseteq V_i$, see \cite[Remark 3.11]{azsp}. We refer the reader to \cite[Theorem 6 and Corollary 8]{got2} for generalization of these facts. Now the following is in order.

\begin{thm}
Let $V,V_1,\ldots, V_n$ be valuations for a field $K$ and $x\in K$. If $V[x]\nsubseteq \bigcup_{i=1}^n V_i$, then there exists $v\in V$ such that $v+x\notin \bigcup_{i=1}^n V_i$.
\end{thm}
\begin{proof}
We may assume that $V_1,\ldots, V_n$ are incomparable and $x\in V_1\cap\cdots\cap V_k$ but $x\notin V_{k+1}\cup\cdots\cup V_n$. If $k=0$, then it suffices to take $v=0$; and if $k=n$, then by Valuation Avoidance Lemma we infer that $V\nsubseteq V_1\cup\cdots\cup V_n$, for otherwise $V\subseteq V_i$ for some $i$ and therefore $V[x]\subseteq V_i$ which is absurd. Thus there exists $v\in V\setminus (V_1\cup\cdots\cup V_n)$ and clearly $v+x\notin V_1\cup\cdots\cup V_n$. Hence suppose that $1\leq k\leq n-1$. We claim that $(V\cap V_{k+1}\cap \cdots\cap V_n)\nsubseteq (V_1\cup\cdots\cup V_k)$. To see this note that if $V\cap V_{k+1}\cap \cdots\cap V_n\subseteq V_1\cup\cdots\cup V_k$, then by \cite[Remark 3.11]{azsp}, either $V\subseteq V_i$ for some $1\leq i\leq k$, and therefore $V[x]\subseteq V_i$ which is impossible or  $V_j\subseteq V_i$ for some $1\leq i\leq k$ and $k+1\leq j\leq n$, which again is impossible since $V_i$'s are incomparable. Thus there exists $v\in (V\cap V_{k+1}\cap \cdots\cap V_n)\nsubseteq (V_1\cup\cdots\cup V_k)$ and therefore $v+x\notin \bigcup_{i=1}^n V_i$.
\end{proof}


\end{document}